%% file: c1alpha.tex
\documentclass[11pt]{amsart}

\input{preambule.tex}
\input{macros.tex}

\title{Minimal H{\"o}lder regularity implying finiteness of integral Menger curvature}
\author{S{\l}awomir Kolasi{\'n}ski 
and Marta Szuma{\'n}ska} 

\address{Institute of Mathematics\\
  University of~Warsaw\\
  Banacha~2, 02-097 Warsaw\\
  Poland}
\email{s.kolasinski@mimuw.edu.pl, m.szumanska@mimuw.edu.pl}
\date{\today}

\begin{document}

\begin{abstract}
  We study two families of integral functionals indexed by a real number
  $p > 0$. One family is defined for $1$-dimensional curves in $\R^3$ and the
  other one is defined for $m$-dimensional manifolds in $\R^n$. These
  functionals are described as integrals of appropriate integrands (strongly
  related to the Menger curvature) raised to power $p$. Given $p > m(m+1)$
  we prove that $C^{1,\alpha}$ regularity of the set (a curve or a manifold),
  with $\alpha > \alpha_0 = 1 - \frac{m(m+1)}p$ implies finiteness of both
  curvature functionals ($m=1$ in the case of curves). We also show that
  $\alpha_0$ is optimal by constructing examples of $C^{1,\alpha_0}$ functions
  with graphs of infinite integral curvature.
\end{abstract}

\maketitle

\input{intro.tex}
\input{prelim.tex}

\input{c1curv.tex}
\input{c1mfld.tex}

\input{curvEx.tex}
\input{mfldEx.tex}

\section*{Acknowledgements}
The first author was partially supported by the Polish Ministry of Science
grant no.~N~N201 611140 (years~2011-2012). \\
The second author was partially supported by the joint German-Polish project "Geometric
curvature energies".\\
The authors are intebted to Prof. P. Strzelecki for his valuable suggestions.
The second author would like to thank Jonas Azzam for his questions, ideas and fruitful discussions.

\bibliography{c1alpha}{}
\bibliographystyle{hplain}

\end{document}

%% file: preambule.tex
\usepackage{a4wide}
\usepackage[T1]{fontenc}
\usepackage{amsmath}
\usepackage{amsfonts}
\usepackage{amssymb}
\usepackage{amsthm}
\usepackage{esint}
\usepackage{graphicx}
\usepackage{enumerate}
\usepackage{mathrsfs}
\usepackage{cite}
\usepackage{wasysym}
\usepackage{xr}
\usepackage{constants}
\usepackage[footnotesize]{caption}
\usepackage{color}

\resetconstant

\newconstantfamily{R}{
symbol=R,
format=\arabic,
}

\newtheorem{thm}{Theorem}[section]
\newtheorem*{thm*}{Theorem}
\newtheorem{prop}[thm]{Proposition}
\newtheorem{lem}[thm]{Lemma}

\theoremstyle{definition}
\newtheorem{defin}[thm]{Definition}
\newtheorem{rem}[thm]{Remark}

%% file: macros.tex
\DeclareMathOperator{\dist}{dist}
\DeclareMathOperator{\diam}{diam}
\DeclareMathOperator{\lin}{span}
\DeclareMathOperator{\opid}{id}
\DeclareMathOperator{\aff}{aff}
\DeclareMathOperator{\conv}{\bigtriangleup}

\DeclareMathOperator{\graph}{graph}
\DeclareMathOperator{\dgras}{\varangle}

\newcommand{\E}{\mathcal{E}}
\newcommand{\M}{\mathcal{M}}

\newcommand{\N}{\mathbb{N}}
\newcommand{\Z}{\mathbb{Z}}

\newcommand{\Ball}{\mathbb{B}}

\newcommand{\R}{\mathbb{R}}
\newcommand{\HM}{\mathcal{H}}

\newcommand{\DC}{\mathcal{K}}

\newcommand{\face}{\mathfrak{fc}}
\newcommand{\height}{\mathfrak{h}}
\newcommand{\plane}{\mathfrak{p}}
\newcommand{\hmin}{\height_{\min}}

\newcommand{\wek}[1]{\vec{#1}}
\newcommand{\pkt}[1]{\mathbf{#1}}
\newcommand{\wpt}[1]{\wek{\pkt{#1}}}

%% file: intro.tex
\section{Introduction}
Geometric curvature energies are functionals defined on submanifolds of an
Euclidean space, whose values can be considered as "total curvature" of the
set. Desired features of such energies are regularization and self-avoidance
effects - i.e. finiteness of the energy implies higher regularity of the set
and excludes self-intersections. These properties make curvature energies
extremely useful in modeling long, entangled objects like DNA molecules,
protein structures or polymer chains; see for example the paper by Banavar et
al.~\cite{MR1966322} or the book by Sutton and Balluffi~\cite{SB97} and the
references therein. In this paper we study two families of such functionals.

Let $S_L$ be a circle of length $L$. Given $p \in [1,\infty]$ and an arc-length
parametrization $\Gamma : S_L \to \R^3$ of some curve $\gamma \subseteq \R^3$
(by \emph{arc-length} we mean that $|\Gamma'| = 1$ a.e.), we define
\begin{displaymath}
  \M_p(\gamma) := \int_{S_L} \int_{S_L} \int_{S_L}
  R^{-p}(\Gamma(x), \Gamma(y), \Gamma(z))\,dx\,dy\,dz \,,
\end{displaymath}
where $R^{-1}(a,b,c)$ is the \emph{Menger curvature} of a triple of points
$(a,b,c)$, i.e. the inverse of the radius of the smallest circle passing
through $a$, $b$ and $c$ (cf. Definition~\ref{menger}).

For an $m$-dimensional manifold $\Sigma \subseteq \R^n$ we set
\begin{displaymath}
  \E_p(\Sigma) = \int_{\Sigma^{m+2}}
  \DC(\pkt{x}_0,\ldots,\pkt{x}_{m+1})^p\ d\HM^m_{\pkt{x}_0} \cdots d\HM^m_{\pkt{x}_{m+1}} \,,
  \qquad
  \Sigma^{m+2} = \underbrace{\Sigma \times \cdots \times \Sigma}_{(m+2) \text{ times}} \,,
\end{displaymath}
where $\DC$ is an appropriately defined analogue of the Menger curvature in
higher dimensions (see Definition~\ref{def:dc}).

The functionals $\M_p$ and $\E_p$ occurred to serve very well as geometric
curvature energies. If~$p > m(m+2)$ then finiteness of $\M_p(\gamma)$ ($m=1$
in this case) or $\E_p(\Sigma)$ implies lack of self-intersections and
$C^{1,\alpha}$ regularity, with $\alpha = 1 - \frac{m(m+2)}p$
(see~\cite{stszvdm} for the $\M_p$ case and~\cite{slawek-phd} for the $\E_p$
case). In this paper we prove that implications in the inverse direction also
hold if we assume a little more about $\alpha$.
\begin{thm*}\ 
  \begin{itemize}
  \item If $\alpha > 1 - \frac 2p$ and $\Gamma \in C^{1,\alpha}$ is injective
    then $\M_p(\gamma)$ is finite.
  \item If $\alpha > 1 - \frac{m(m+1)}p$ and $\Sigma$ is compact and
    $C^{1,\alpha}$ regular then $\E_p(\Sigma)$ is finite.
  \end{itemize}
  Moreover, $\alpha_0 = 1-\frac{m(m+1)}p$ is the minimal H{\"o}lder exponent
  above which we have finite energy.
\end{thm*}

Note that there is a gap between $1-\frac{m(m+2)}p$ and $1-\frac{m(m+1)}p$.
To~show that our estimates are sharp, in~\S\ref{sec:curv-ex} we construct
a~concrete example of a~$C^{1,1-2/p}$ function $F : \R \to \R$ for which
$\M_p(\graph(F))$ is infinite. Then in~\S\ref{sec:mfld-ex}, using essentially
the same construction, we also cook~up an~example of a~function $G : \R^m \to
\R$ for which $\E_p(G([0,1]^m))$ is infinite.

Recently Blatt \cite{blatt-note} showed that finiteness of $\M_p(\gamma)$ is
exactly equivalent to the condition that $\gamma$ lies in a
Sobolev-Slobodeckij space $W^{1+s,p}$, where $s = 1 - \frac 2p$. For
$\alpha>1-\frac 2p$ we have $C^{1,\alpha} \subseteq W^{1+s,p} \subseteq
C^{1,1-3/p}$, so the result by Blatt generalizes \cite{stszvdm} and also some
results of this paper. Nevertheless, our method is simpler and more
geometrical and we apply it also in higher dimensions.

Besides $\M_p$, in~\cite{pre05199464}, \cite{MR2489022} and~\cite{stszvdm}
two other functionals.
\[
\mathcal{U}_p:= \int_{S_L} \left(\inf_{\{s,t \in S_L\setminus\{u\} \ |  \ s\neq t\} }R(\Gamma(s),\Gamma(t),\Gamma(u) \right)^{-p} du
\]
and
\[
\mathcal{I}_p:= \int_{S_L} \int_{S_L} \left(\inf_{s \in S_L\setminus\{t,u\} }R(\Gamma(s),\Gamma(t),\Gamma(u) \right)^{-p} du \, dt \,,
\]
were examined. In all cases, finiteness of the above functionals (for
sufficiently large $p$) implies existence of injective arc-length
parametrization of $\gamma$ and assures its higher regularity.

Yet another example of a curvature energy for curves is the tangent-point
energy defined as
\[
\E_q(\Gamma) := \int_0^L \int_0^L \frac{ds\,dt}{r^q(\Gamma(t),\Gamma(s)},
\]
where $\Gamma$ is an arc-length parametrization of the curve, and
$r(\Gamma(t), \Gamma(s))$ is the radius of the unique circle passing through
$\Gamma(s)$ and tangent to the curve at $\Gamma(t)$. Properties of curves with
finite $\E_q$ energies for $q>2$ were investigated in \cite{1014.57007}
(in~$C^2$ case) and in \cite{strzvdM} (in continuous case). A~similar
tangent-point energy in higher dimensions was studied in~\cite{1102.3642},
where the authors once again establish self-avoidance and smoothing effects.

The functional $\M_2$ proved to be useful also in harmonic analysis. Using
roughly the same formula, one can define $\M_p$ for any Borel set $E$. David
and L{\'e}ger~\cite{MR1709304} showed that $1$-dimensional sets with finite
$\M_2$ total curvature are $1$-rectifiable. This was a crucial step in the
proof of Vitushkin's conjecture and allowed to fully characterize removable
sets of bounded analytical functions. Surveys of Mattila~\cite{MR1648114} and
Tolsa~\cite{MR2275656} explain in more detail the connection between these
subjects.

A close analogue of the energy $\E_p$, defined for $2$-dimensional, non-smooth
surfaces in $\R^3$ was also studied by Strzelecki and von der Mosel
in~\cite{0911.2095}. The authors proved that one can use their notion of total
curvature to impose topological constraints in variational problems. They
proved existence of area minimizing surfaces in a given isotopy class under
the constraint of bounded curvature.

Lerman and Whitehouse in~\cite{0805.1425} and in~\cite{MR2558685} suggested a
whole class of curvature energies for higher dimensional objects. However,
the~integrands in their definitions scale differently and it seems that these
energies can not serve our needs. Nevertheless, the authors proved
\cite[Theorems 1.2 and 1.3]{MR2558685} that their integral curvatures can be
used to characterize $d$-dimensional rectifiable measures thereby establishing
a~link between the theory of geometric energies and uniform rectifiablility in
the sense of David and Semmes~\cite{MR1251061}.

\begin{rem}
  We shall use a letter $C$ to denote a general constant, whose value may
  change from line to line even in one series of transformations.
\end{rem}

% Local Variables:
% coding: iso-8859-2
% eval: (ispell-change-dictionary "american")
% mode: latex
% mode: flyspell
% End:

% LocalWords:  tuple submanifolds injective functionals Menger Sobolev Blatt
% LocalWords:  Slobodeckij Strzelecki Mosel tuples Borel minimizers isotopy
% LocalWords:  rectifiablility

%% file: prelim.tex
\section{Preliminaries}

In the section we introduce the notation that will be used throughout the
paper. We also explain the relations between two types of energies we
consider.

We use standard symbols for commonly used notions, therefore $\Ball(x,r)$ stands for the 
ball with radius $r$ centered at $x$, while we write $\Ball_r$ if the origin is the center
of the ball and $\Ball_r^m$, when we additionally want to emphasize the dimension of the ball.
We denote m-dimensional
Hausdorff measure by $\HM^m$ and $T_a\Sigma$ denotes the vector space tangent to $\Sigma$
at the point $a$. The symbol $S_L$ stands for the circle of length $L$,
i.e. $S_L = \R/L\Z$.

For a tuple~$T = (\pkt{x}_0,\ldots,\pkt{x}_k)$ of~$(k+1)$ points in~$\R^n$, we~write
$\conv T$ to denote the convex hull of the set $\{ \pkt{x}_0,\pkt{x}_1,\ldots,
\pkt{x}_k \}$, i.e. the smallest convex subset of~$\R^n$ which contains all
the points $\pkt{x}_0$, \ldots, $\pkt{x}_k$.  Typically $\conv T$ will just be
$k$-simplex (a~triangle for~$k = 2$ and a~tetrahedron for~$k = 3$).

As we mentioned in the introduction, Menger curvature of three points is a
reciprocal of the radius of the smallest circle passing through those points.
The expression given below, can be treated as an equivalent definition of the
radius of the curvature of three distinct points.
\begin{defin}
  \label{menger}
  Let $x_0$,$x_1$,$x_2$ be three points in $\R^n$. We define Radius of Menger curvature of $x_0$,$x_1$,$x_2$ as
  \begin{displaymath}
    R(x_0,x_1,x_2) = \frac{|x_1 - x_0||x_2 - x_0||x_2 - x_1|}{4 \HM^{2}(\conv(x_0,x_1,x_2))} \,.
  \end{displaymath}
\end{defin}
We can use the above formula to compare Menger curvature and its higher
dimensional generalizations \footnote{Natural generalization of Menger
  curvature of three points would be the inverse of the radius of the sphere
  passing through four points. However this definition is not rewarding for
  integral curvature energies - see~\cite[Appendix B]{0911.2095}}.
   In this
paper we use integral curvature functional $\E_p$ defined in~\cite{slawek-phd}
whose integrand is the $p$-th power of the discrete curvature $\DC$.
\begin{defin}
  \label{def:dc}
  Let $T = (\pkt{x}_0,\ldots,\pkt{x}_{m+1})$ be an $(m+2)$-tuple of points in
  $\R^n$. The \emph{discrete curvature} of $T$ is given by the formula
  \begin{displaymath}
    \DC(T) = \frac{\HM^{m+1}(\conv T)}{\diam(T)^{m+2}} \,.
  \end{displaymath}
\end{defin}
Now we introduce the definition of our functional.
\begin{defin}
  \label{def:p-energy}
  Let $\Sigma \subseteq \R^n$ be some $m$-dimensional subset of $\R^n$. We define
  the \emph{$p$-energy} of $\Sigma$ by the formula
  \begin{displaymath}
    \E_p(\Sigma) = \int_{\Sigma^{m+2}} \DC(\pkt{x}_0,\ldots,\pkt{x}_{m+1})^p\ 
    d\HM^m_{\pkt{x}_0} \cdots d\HM^m_{\pkt{x}_{m+1}} \,.
  \end{displaymath}
\end{defin}
The quantity $\DC(T)$ should be seen as a~generalization of the Menger
curvature to higher dimensions. It behaves in the same way as $R^{-1}$ under
scaling, i.e.
\begin{displaymath}
  \DC(\lambda T) = \lambda^{-1} \DC(T) \,.
\end{displaymath}
Notice that we always have $R^{-1}(x,y,z) > \DC(x,y,z)$. Furthermore, for a
class of roughly regular triangles $T=(x,y,z)$ (i.e. satisfying $\hmin(T) \ge
\eta \diam(T)$, where $\hmin(T)$ is the minimal height of $T$ and $\eta \in
(0,1]$ is some fixed number) the two quantities $\DC(T)$ and $R^{-1}(T)$ are
comparable up~to a~constant depending only on~$\eta$. However they are not
comparable, when considered on the family of all triangles so we cannot infer
finiteness of $\M_p(\gamma)$ from finiteness of $\E_p(\gamma)$.

The definition of $\DC$ is based on another notion of discrete curvature
$\DC_{\text{SvdM}}$ introduced by Strzelecki and von der Mosel
in~\cite{0911.2095} for $4$-tuples of points (tetrahedrons). Yet again,
the~quantities $\DC(T)$ and $\DC_{\text{SvdM}}(T)$ are comparable for the
class of roughly regular tetrahedrons, i.e. such~that $\hmin(T) \ge \eta
\diam(T)$ for some fixed $\eta \in (0,1]$.

In the proofs we will use the Jones' $\beta$-numbers. For a set $E \subset
\R^n$, for any $x \in \R ^n$ and $r \in \R $ we define
\[
\beta_E^m(x,r) := \inf_{H} \sup_{y \in E \cap \Ball(x,r)}\frac{d(y,x+H)}{r},
\]
where the infimum is taken over all $m$-hyperplanes $H$ in the
Grassmannian~$G(m,n)$. The quantity $\beta_E^m(x,t)$ measures, in a scale
invariant way, how well the set is approximated by hyperplanes in the ball
$\Ball(x,t)$. We omit the indices $E$ and $m$ if the choice of the set and its
dimension is clear from the context. The~relations between total Menger
curvature ($\M_2$) and double integral of $\beta$-numbers for rectifiable
curves and Ahlfors regular sets were investigated by Peter Jones, who never
published those results, however they are presented in Herve Pajot's book
\cite[chapter 3]{Pajot}.

% Local Variables:
% coding: iso-8859-2
% eval: (ispell-change-dictionary "american")
% mode: latex
% mode: flyspell
% End:

% LocalWords:  tuple Menger infimum Grassmannian Ahlfors Pajot Strzelecki Mosel
% LocalWords:  tuples indices

%% file: c1curv.tex
\section{Injective $C^{1,\alpha}$ curves have finite $\M_p$ curvature}

 The main step of the proof of the finiteness of $\M_p(\gamma)$ for  $C^{1,\alpha}$ curves
with $\alpha>1-2/p$ , is to show that one can control the angle between secants on small arcs
on the curves. This leads to the estimation of Menger curvature. 
\begin{thm}
  \label{converse}
  If $\Gamma: S_L \to \R^3$ is an injective arc-length parametrization of
  $\gamma$ and $\Gamma\in C^{1,\alpha}(S_L)$, where $\alpha>1-\frac{2}{p}$ for
  some $p>2$, then $\M_p(\gamma)$ is finite.
  %\begin{displaymath}
   % \M_p(\gamma) = \int_{S_L} \int_{S_L} \int_{S_L} R^{-p}(\Gamma(x), \Gamma(y), \Gamma(z))\ dx\ dy\ dz.
  %\end{displaymath} 
\end{thm}
We start with an easy lemma, which shows that the parametrization is bi-lipschitz. 
\begin{lem}
\label{bilip}
If $\Gamma: S_L \to \R^3$ is injective arc-length parametrization of $\gamma$, such that 
\[
|\Gamma(x) - \Gamma(y)| \le C |x-y|^\alpha,
\] 
then there exists a constant 
$\lambda = \lambda(C,\alpha,L) $, such that
\[
|\Gamma(x) - \Gamma(y)| > \lambda |x-y|
\]
\end{lem}
For the convenience of the reader
we repeat the proof, which can be found in \cite{MR2489022}.
\begin{proof} 
First we show that $\Gamma$ is uniformly locally bi-lipschitz, i.e.
for given $c_1\in(0,1)$
\[
\exists_{\delta>0} \ \textrm{such that} \ \forall_{x,y \in S_L} \ |x-y| < \delta \   \Rightarrow
\ |\Gamma(x) - \Gamma(y)| > c_1 |x-y| 
\]
We take $\delta := \Big(\frac{1-c_1}{C}\Big)^{1/\alpha}$, then for all $|x-y| < \delta$
\[
|\Gamma'_1(x) - \Gamma'_1(y)| < |\Gamma'(x) - \Gamma'(y)| < C |x-y|^\alpha < 1-c_1.
\]
Without loss of generality we can assume that $\Gamma'(x) = [1,0,0]$ and for $t$ satisfying $|x-t| <\delta $ we have
\[
\Gamma'_1(x) - |\Gamma'_1(x) - \Gamma'_1(t)| > c_1 \quad \textrm{and \ thus} \quad 
\Gamma'_1(t)> c_1. 
\]
Using absolute continuity of $\Gamma$ we get
\[
|\Gamma(x) - \Gamma(y)| = \big|\int_x^y \Gamma'(t) dt \big| \ge \big|\int_x^y \Gamma'_1(t) dt \big|> c_1 |x-y| 
\]
To finish the proof it is enough to notice that the set
\[
A_\delta := \{(x,y) \in S_L \ | \ |x-y| \ge \delta\}
\]
is compact, thus injectivity and continuity of $\Gamma$ implies that $\inf_{A_\delta}|\Gamma(x) - \Gamma(y)| = a >0 $, therefore
for $|x-y| \ge \delta$ we have 
\[
|\Gamma(x) - \Gamma(y)| \ge \frac{a}{L}|x-y|.
\]
\end{proof}

To formulate next the lemma we need to introduce geometric objects we will use.  Let $C^+(P,\wek{v},\alpha)$ be a "half cone" with a
vertex P, an axis parallel to the given vector $\wek{v}$ and an opening angle
$\alpha$
\begin{displaymath}
  C^+(P,\wek{v},\alpha) := \left\{ P+x  :  |\varangle(\wek{v},x)| < \frac{\alpha}{2} \right\} \,.
\end{displaymath}
Intersection of two half-cones with vertices $P$ and $Q$ and common axis $PQ$
and opening angle~$\alpha$ (see fig. 1)
will be denoted as follows
\begin{displaymath}
  D(P,Q,\alpha) := C^+(P,\wek{PQ}, \alpha) \cap C^+(Q,\wek{QP}, \alpha)
\end{displaymath}  
\begin{center}
  \includegraphics[width=0.5\textwidth]{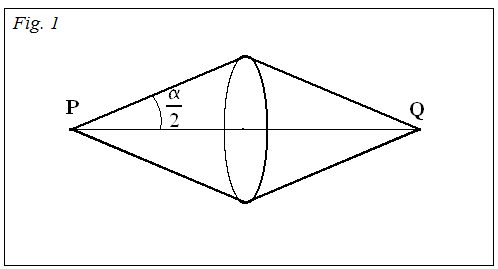}
\end{center}

Now we are ready to formulate and prove the lemma.
\begin{lem}
  \label{holder}
  Let $\Gamma : S_L \to \R^3 $ be injective,  and 
\[
|\Gamma'(x) - \Gamma'(y)| \le C |x-y|^\alpha \quad \textrm{for all} \quad x,y \in S_L
\]
 Then for any $x,y \in S_L$ satisfying $\frac{5}{2}C|x-y|^\alpha <
  1$ the following inclusion holds
  \begin{displaymath}
    \Gamma([x,y]) \subset D(\Gamma(x),\Gamma(y), \eta|x-y|^\alpha),
  \end{displaymath}
  where $\eta$ is a constant which depends only on $C$.
\end{lem} 

\begin{proof}[Proof of Lemma \ref{holder}]
% As $\Gamma \in C^{1,\alpha}(S_L)$, there exists a constant $C$ such that
% \begin{displaymath}
%   |\Gamma'(x) - \Gamma'(y)| \le C|x-y|^\alpha.
% \end{displaymath}
  Let $x,y,z \in S_L$ be such that $x < z < y$. We are going to estimate the
  angle between two secant lines: one passing through $\Gamma(x)$ and
  $\Gamma(y)$ and the other one passing through $\Gamma(x)$ and $\Gamma(z)$.
  First, note that for any two unit vectors $u,v \in S^{n-1}$ forming a small
  angle $\varangle(u,v) = \theta$ we have $\theta \simeq |u-v|$. Hence, it
  suffices to estimate the difference $\frac{\Gamma(x) - \Gamma(y)}{|\Gamma(x)
    - \Gamma(y)|}$ and $\frac{\Gamma(x) - \Gamma(z)}{|\Gamma(x) -
    \Gamma(z)|}$. Let us calculate
  \begin{multline*}
    \left|
      \frac{\Gamma(x) - \Gamma(y)}{|\Gamma(x) - \Gamma(y)|}
      - \frac{\Gamma(x) - \Gamma(z)}{|\Gamma(x) - \Gamma(z)|}
    \right|
    \le
    \left|\frac{\Gamma(x) - \Gamma(y)}{|\Gamma(x) - \Gamma(y)|} - \Gamma'(x)\right| \\
    + \left|\frac{\Gamma(x) - \Gamma(z)}{|\Gamma(x) - \Gamma(z)|} -\Gamma'(z)\right|
    + |\Gamma'(x) - \Gamma'(z)| \,.
  \end{multline*}
  We have
  \begin{align*}
    \left|\Gamma'(x)- \frac{\Gamma(x) - \Gamma(y)}{|\Gamma(x) - \Gamma(y)|} \right|
    &\le \left|\Gamma'(x) - \frac{\Gamma(x) - \Gamma(y)}{|x-y|}\right|  
    +  \left|\frac{\Gamma(x) - \Gamma(y)}{|x-y|} - \frac{\Gamma(x) - \Gamma(y)}{|\Gamma(x) - \Gamma(y)|}\right| \\
    &\le  2 \left|\Gamma'(x) - \frac{\Gamma(x) - \Gamma(y)}{|x-y|}\right| \,,
  \end{align*}
  where the last inequality holds, because the distance between vector $v =
  \frac{\Gamma(x) - \Gamma(y)}{|x-y|}$ and its projection onto the unit sphere
  $\frac{\Gamma(x) - \Gamma(y)}{|\Gamma(x) - \Gamma(y)|}$ is not greater than
  the~distance between $v$ and any arbitrarily chosen unit vector (recall that
  $\Gamma$ is an arc-length parametrization, so $|\Gamma'| \equiv 1$).
  \begin{align*}
    \left|\Gamma'(x) - \frac{\Gamma(x) - \Gamma(y)}{|x-y|}\right| 
    &= \left| \Gamma'(x) - \frac{1}{y - x} \int_{[x,y]} \Gamma'(s)\ ds\right| \\
    &\le \frac{1}{y - x} \int_{[x,y]}| \Gamma'(x) - \Gamma'(s)|\ ds \le \frac{1}{y - x} \int_{[x,y]} C(s-x)^\alpha\ ds \\ 
    &\le C(y-x)^\alpha \,.
  \end{align*}
  Combining the above inequalities we obtain
  \begin{displaymath}
    \left|
      \frac{\Gamma(x) - \Gamma(y)}{|\Gamma(x) - \Gamma(y)|}
      - \frac{\Gamma(x) - \Gamma(z)}{|\Gamma(x) - \Gamma(z)|}
    \right| 
    \le 2C|x-y|^\alpha + 2C|x-z|^\alpha + C|x-z|^\alpha \le 5C|x-y|^\alpha \,.
  \end{displaymath} 
  This implies that $\Gamma[(x,y)]$ is included in the cone 
  \begin{displaymath}
    C\left(
      \Gamma(x), \Gamma(y) - \Gamma(x),2\arcsin\big(\tfrac 52 C|x-y|^\alpha\big)
    \right) \,.
  \end{displaymath}
  Analogously 
  \begin{displaymath}
    \Gamma[(x,y)]\in C
    \left(
      \Gamma(y), \Gamma(x) - \Gamma(y),2\arcsin\big(\tfrac 52 C|x-y|^\alpha\big)
    \right) \,.
  \end{displaymath}
  Thus 
  \begin{displaymath}
    \Gamma[(x,y)]\in D\left(\Gamma(x),\Gamma(y),\eta|x-y|^\alpha \right) \,.
  \end{displaymath}
\end{proof}

The lemma proven above gives an estimation for Jones $\beta$-numbers; it is easy to notice that there 
exists $R_0$, such that for $r<R_0$ we have
\[
\beta(x,r) < r^{1+\alpha}.
\]
In case of plane curves it is possible to use the estimation to prove the finiteness of $\M_p$,
(one can modify the reasoning from \cite{Pajot}), but it is not clear if it is possible to adapt it to curves in 
$\R ^3$ and even in 2-space the arguing is long and complicated. Here, Lemma \ref{holder} gives us  additional information. We not only know that the curve is close to a line in a small ball, but we can also point out the line.
This information makes the proof of the finiteness of $\M_p$ much easier, as the following lemma holds. 

\begin{lem}
  \label{top}
  Let $\Gamma:S_L \to \R^3 $ be an injective, arc-length parametrization of a
  curve $\gamma$ and let $p>2$ and $\alpha > 1-\frac {2}{p}$. If there exist
  constants $\eta$ and $\varepsilon>0 $ such that $\eta \varepsilon^\alpha<\frac{\pi}{2}$ and for each $|x-y|<\varepsilon$ we have
  \begin{equation}
  \label{cukierek}
    \Gamma((x,y)) \subset D(\Gamma(x), \Gamma(y), \eta |x - y|^{\alpha}) \,,
  \end{equation}
  then $\M_p(\gamma)<\infty$.
\end{lem}

\begin{proof}[Proof of Lemma \ref{top}]
  We start with a simple geometric observation. If $S\in D(P,Q,\beta)$  and $0< \beta <\frac{\pi}{2}$ then
  \begin{displaymath}
    c(P,Q,S) = \frac 1{R(P,Q,S)} \le \frac{2\sin{\beta}}{|PQ|}.
  \end{displaymath}
 Thus, from \eqref{cukierek}, for $|x-y|<\varepsilon$ and $z\in [x,y]$
  we have 
  \begin{displaymath}
    c(\Gamma(x),\Gamma(y),\Gamma(z)) 
    \le \frac{2\sin\left(\eta |x - y|^\alpha \right)}{|\Gamma(x)-\Gamma(y)|} \le  \frac{2\eta |x - y|^\alpha}{|\Gamma(x)-\Gamma(y)|}
  \end{displaymath}
 
  As we know from Lemma \ref{bilip}, $\Gamma$ is bi-lipschitz, hence there exists a constant $d>0$
  such that
  \begin{equation}
    \label{oszac}
    c(\Gamma(x),\Gamma(y),\Gamma(z)) \le d |x-y|^{\alpha - 1}.
  \end{equation}
  Now we are ready to estimate the triple integral
  \begin{align*}
    \M_p(\gamma) &= \int_{S_L} \int_{S_L} \int_{S_L} c^p(\Gamma(x),\Gamma(y), \Gamma(z))\ dx\ dy\ dz \\
    &\le C \int_{S_L} \int_{\{y\in S_L \ | \ |x-y|<\varepsilon\}} \int_{[x,y]} c^p(\Gamma(x),\Gamma(y), \Gamma(z))\ dx\ dy\ dz \\
    &+ \int_{S_L} \int_{\{y\in S_L \ | \ |x-y|>\varepsilon\}} \int_{S_L} c^p(\Gamma(x),\Gamma(y), \Gamma(z))\ dx\ dy\ dz \\
    &\stackrel{\eqref{oszac}}\le C \int_{S_L} \int_{\{y\in S_L \ | \ |x-y|<\varepsilon\}} |x-y|^{p({\alpha-1})}\cdot|x-y|\ dx\ dy \\ 
    &+ \int_{S_L} \int_{\{y\in S_L \ | \ |x-y|>\varepsilon\}} \int_{S_L} \varepsilon^p\ dx\ dy\ dz \\
    &\le C \int_{S_L} \int_{\{y\in S_L \ | \ |x-y|<\varepsilon\}} |x-y|^{p\alpha - p +1}\ dx\ dy + \textrm{const} 
    < \infty \,,
  \end{align*}
  as $\alpha > 1-\frac{2}{p}$.
\end{proof}

The proof of Theorem \ref{converse} follows immediately from the Lemmas \ref{holder} and 
\ref{top}.

\begin{proof}[Proof of Theorem \ref{converse}]
  As $\Gamma$ satisfies assumption of Lemma \ref{holder},
  we know that there exists $\varepsilon>0$ such that if $|x-y|\le \varepsilon$ then
  \begin{displaymath}
    \Gamma([x,y]) \subset D(\Gamma(x),\Gamma(y), {\rm d}|x-y|^\alpha) \, ,
  \end{displaymath}
  where $d$ is a constant which depends only on H{\"o}lder constant $C$.
  Thus using Lemma \ref{top} we obtain the thesis.
\end{proof}

% Local Variables:
% coding: iso-8859-2
% eval: (ispell-change-dictionary "american")
% mode: latex
% mode: flyspell
% End:

% LocalWords:  injectivity injective Lipschitz vertices

%% file: c1mfld.tex
\section{Manifolds of class $C^{1,\alpha}$ have finite integral curvature}
In this section we prove a counterpart of Theorem~\ref{converse} for
$m$-dimensional submanifolds of $\R^n$.
\begin{thm}
  \label{thm:fin-curv}
  Let $p > m(m+1)$ be some number and let $\Sigma \subseteq \R^n$ be a compact
  manifold of class $C^{1,\alpha}$. If $\alpha > 1 - \frac{m(m+1)}p$ then
  $\E_p(\Sigma)$ is finite.
\end{thm}

\begin{lem}
  \label{lem:beta-est}
  Let $\Sigma \subseteq \R^n$ be a compact manifold of class $C^{1,\alpha}$ for
  some $\alpha \in (0,1)$. Then there exist constants $R = R(\Sigma) > 0$ and $C
  = C(\Sigma) > 0$ such that for each $a \in \Sigma$ and each $r \le R$
  \begin{displaymath}
    \beta(a,r) \le C r^{\alpha} \,.
  \end{displaymath}
\end{lem}

\begin{proof}
  Since $\Sigma$ is compact, we can find a radius $R > 0$ and a constant $C > 0$
  such that for each $a \in \Sigma$ there exists a function $f_a \in
  C^{1,\alpha}(T_a\Sigma \cap \Ball_{2R},T_a\Sigma^{\perp})$ such that
  \begin{displaymath}
    \Sigma \cap \Ball(a,R) \subseteq a + \graph(f_a) \,,
  \end{displaymath}
  \begin{displaymath}
    f_a(0) = 0 \,, \quad Df_a(0) = 0 \,,
  \end{displaymath}
  \begin{displaymath}
    \text{and} \quad \forall x,y \in \Ball^m_{2R} \quad |Df_a(x) - Df_a(y)| \le C |x-y|^{\alpha} \,.
  \end{displaymath}
  Fix some $a \in \Sigma$ and a radius $r \le R$. Let $b \in \Sigma \cap
  \Ball(a,r)$. Since $\Sigma \cap \Ball(a,R)$ is the graph of $f$, there exists
  a point $x \in T_a\Sigma$ such that $b = a + x + f_a(x)$. By the fundamental
  theorem of calculus we have
  \begin{align*}
    |f_a(x)| = |f_a(x) - f_a(0)| &= \left| \int_0^1 \tfrac{d}{dt}f_a(tx)\ dt \right| \\
    &\le |x| \sup_{y \in T_a\Sigma \cap \Ball_{|x|}} |Df_a(y) - Df_a(0)| \\
    &\le C |x|^{1 + \alpha} \le C |b-a|^{1 + \alpha} \,.
  \end{align*}
  Note that $|f_a(x)|$ is just the distance of $b$ from the affine plane $a +
  T_a\Sigma$. Hence
  \begin{displaymath}
    \sup_{b \in \Sigma \cap \Ball(a,r)} \dist(b, a + T_a\Sigma) \le C r^{1+\alpha}
  \end{displaymath}
  and we obtain
  \begin{align*}
    \beta(a,r) &= \frac 1r \inf_{H \in G(n,m)} \Big( \sup_{b \in \Sigma \cap \Ball(a,r)} \dist(b, a + H) \Big) \\
    &\le \frac 1r \sup_{b \in \Sigma \cap \Ball(a,r)} \dist(b, a + T_a\Sigma)
    \le C r^{\alpha} \,.
  \end{align*}
\end{proof}

\begin{lem}
  \label{lem:dc-est}
  Let $\Sigma \subseteq \R^n$ be an $m$-dimensional manifold. Choose $m+2$
  points $\pkt{x}_0$,\ldots,$\pkt{x}_{m+1}$ of $\Sigma$ and set $T =
  \conv(\pkt{x}_0,\ldots,\pkt{x}_{m+1})$ and $d = \diam(T)$. There exists a
  constant $C = C(m,n)$ such that
  \begin{displaymath}
    \HM^{m+1}(T) \le C \beta(\pkt{x}_0,d) d^{m+1} \,,
  \end{displaymath}
  hence
  \begin{displaymath}
    \DC(\pkt{x}_0,\ldots,\pkt{x}_{m+1}) \le C \frac{\beta(\pkt{x}_0,d)}{d} \,.
  \end{displaymath}
\end{lem}

Note that we assumed $\Sigma$ to be a manifold but the proof works also for an
arbitrary set $\Sigma$ of Hausdorff dimension $m$ or even for any set.

\begin{proof}
  If the affine space $\aff\{\pkt{x}_0, \ldots, \pkt{x}_{m+1}\}$ is not $(m+1)$-dimensional
  then $\HM^{m+1}(T) = 0$ and there is nothing to prove. Hence, we can assume
  that $T$ is an $(m+1)$-dimensional simplex. The measure $\HM^{m+1}(T)$ can be
  expressed by the formula
  \begin{displaymath}
    \HM^{m+1}(T) 
    = \frac{1}{m+1} \dist(\pkt{x}_{m+1}, \aff\{\pkt{x}_0,\ldots,\pkt{x}_m\})
    \HM^m(\conv(\pkt{x}_0,\ldots,\pkt{x}_m)) \,.
  \end{displaymath}
  In the same way, one can express the measure $\HM^m(\conv(\pkt{x}_0,\ldots,\pkt{x}_m))$, so
  certainly
  \begin{displaymath}
    \HM^{m+1}(T) \le \frac 1{(m+1)!} d^{m+1} \,.
  \end{displaymath}
  Hence, if $\beta(\pkt{x}_0,d) = 1$, then there is nothing to prove, so we can assume
  that $\beta(\pkt{x}_0,d) < 1$.

  Due to compactness of the Grassmannian $G(n,m)$ we can find an $m$-plane $H
  \in G(n,m)$ such that
  \begin{equation}
    \label{eq:points-dist}
    \sup_{y \in \Sigma \cap \Ball(\pkt{x}_0,d)} \dist(y, \pkt{x}_0 + H) = d \beta(\pkt{x}_0,d) \,.
  \end{equation}
  Set $h = d \beta(\pkt{x}_0,d) < d$. Without loss of generality we can assume that
  $\pkt{x}_0$ lies at the origin. Let us choose an orthonormal coordinate system
  $v_1$, \ldots, $v_n$ such that $H = \lin\{v_1, \ldots, v_m\}$. Because
  of~\eqref{eq:points-dist} in our coordinate system we have
  \begin{displaymath}
    T \subseteq [-d, d]^m \times [-h,h]^{n-m}\,.
  \end{displaymath}
  Of course $T$ lies in some $(m+1)$-dimensional section of the above
  product. Let
  \begin{align*}
    V &:= \aff \{ \pkt{x}_0, \ldots, \pkt{x}_{m+1} \} = \lin \{ \pkt{x}_1, \ldots, \pkt{x}_{m+1} \} \,,\\
    Q(a,b) &:= [-a,a]^m \times [-b,b]^{n-m} \,,\\
    Q &:= Q(d,h) \\
    \text{and} \quad
    P &:= V \cap Q \,.
  \end{align*}
  Note that all of the sets $V$, $Q$ and $P$ contain $T$. Choose another
  orthonormal basis $w_1$, \ldots, $w_n$ of $\R^n$, such that $V = \lin\{w_1,
  \ldots, w_{m+1}\}$. Set 
  \begin{displaymath}
    S := \{ x \in V^{\perp} : |\langle x, w_i \rangle| \le h \text{ for } i=1,\ldots,m+1 \} \,.
  \end{displaymath}
  Observe that $S$ is just the cube $[-h,h]^{n-m-1}$ placed in the orthogonal
  complement of $V$ and that $\diam S = 2 h \sqrt{n-m-1} =: 2Ah$, where $A =
  A(n,m) = \diam([0,1]^{n-m-1})$. In this setting we have
  \begin{equation}
    \label{eq:PxS}
    P \times S 
    = P + S 
    \subseteq Q(d + 2h A, h + 2hA) \,.
  \end{equation}
  Recall that $h = d \beta(\pkt{x}_0,d) < d$. We obtain the following estimate
  \begin{align*}
    \HM^n(T \times S) 
    &\le \HM^n(P \times S) 
    \le \HM^n(Q(d + 2h A, h + 2 h A))  \\
    &\le (2 d + 4 h A)^m (2 h + 4 h A)^{n-m}  \\
    &\le (2 d + 4 d \beta(\pkt{x}_0,d) A)^m
    (2 d \beta(\pkt{x}_0,d) + 4 d \beta(\pkt{x}_0,d) A)^{n-m} \\
    &\le (2 + 4 A)^n d^n \beta(\pkt{x}_0,d)^{n-m}  \,.
  \end{align*}
  On the other hand we have
  \begin{align*}
    \HM^n(T \times S) &= \HM^{m+1}(T) \HM^{n-m-1}(S) 
    = \HM^{m+1}(T) 2^{n-m-1} h^{n-m-1} \\
    &= 2^{n-m-1} \HM^{m+1}(T) d^{n-m-1} \beta(\pkt{x}_0,d)^{n-m-1} \,.
  \end{align*}
  Hence
  \begin{align*}
    2^{n-m-1} \HM^{m+1}(T) d^{n-m-1} \beta(\pkt{x}_0,d)^{n-m-1}
    &\le (2 + 4 A)^n d^n \beta(\pkt{x}_0,d)^{n-m}
    \quad \iff \\ \iff \quad
    \HM^{m+1}(T) &\le (2 + 4 A)^n 2^{-(n-m-1)} d^{m+1} \beta(\pkt{x}_0,d) \,.
  \end{align*}
  We may set $C = C(n,m) = (2 + 4 A)^n 2^{-(n-m-1)}$.
\end{proof}

Now we can prove the main result of this section.
\begin{proof}[Proof of Theorem~\ref{thm:fin-curv}]
  Let 
  \begin{displaymath}
    \mu = \underbrace{\HM^m \otimes \cdots \otimes \HM^m}_{m+1} \,.
  \end{displaymath}
  If $T = (x_0,x_1,\ldots,x_{m+1}) \in \Sigma^{m+2}$, we shall write $T =
  (x_0,\bar{x})$. Using Lemma~\ref{lem:dc-est} we obtain
  \begin{align*}
    \E_p(\Sigma) &= \int_{\Sigma^{m+2}} \DC(x_0,\bar{x})^p\ d\HM^m(x_0)\  d\mu(\bar{x}) \\
    &\le \int_{\Sigma} \int_{\Sigma^{m+1}} \left(
      \frac{\beta(x_0,\diam(x_0,\ldots,x_{m+1}))}{\diam(x_0,\ldots,x_{m+1})}
    \right)^p\ d\HM^m(x_0) d\mu(\bar{x}) \,.
  \end{align*}
  For $x_0 \in \Sigma$ and $k \in \N$ we define the sets
  \begin{displaymath}
    \Sigma_k(x_0) := \{ (x_1,\ldots,x_{m+1}) \in \Sigma^{m+1} : \diam(x_0,\ldots,x_{m+1}) \in (2^{-k-1},2^{-k}] \} \,.
  \end{displaymath}
  Choose $K_0 \in \Z$ such that $2^{-K_0} \ge 2 \diam(\Sigma)$. Now we can
  write
  \begin{displaymath}
    \E_p(\Sigma) \le
    \int_{\Sigma} \sum_{k=K_0}^{\infty}
    \int_{\Sigma_k(x_0)} \left(
      \frac{\beta(x_0,\diam(x_0,\bar{x}))}{\diam(x_0,\bar{x})}
    \right)^p\ d\HM^m(x_0) d\mu(\bar{x}) \,.
  \end{displaymath}
  Fix some small number $\varepsilon > 0$. Since $\Sigma$ is compact, we can
  find a radius $R > 0$ and a constant $C > 0$ such that for each $a \in \Sigma$
  there exists a function $f_a \in C^{1,\alpha}(T_a\Sigma \cap
  \Ball_{2R},T_a\Sigma^{\perp})$ such that
  \begin{displaymath}
    \Sigma \cap \Ball(a,R) \subseteq a + \graph(f_a) \,,
  \end{displaymath}
  \begin{displaymath}
    f_a(0) = 0 \,, \quad Df_a(0) = 0 
  \end{displaymath}
  \begin{displaymath}
    \text{and} \quad \forall x,y \in \Ball^m_{2R} \quad |f_a(x) - f_a(y)| \le \varepsilon |x-y| \,.
  \end{displaymath}
  For $r < R$ and $x_0 \in \Sigma$ we have the following estimate
  \begin{displaymath}
    \HM^m(\Sigma \cap \Ball(x_0,r))
    \le (1+\varepsilon)^m \HM^m((x_0 + T_{x_0}\Sigma) \cap \Ball(x_0,r))
    = (1+\varepsilon)^m \omega_m r^m \,.
  \end{displaymath}
  Choose $k_0 \in \Z$ such that $2^{-k_0} \le R$. Then for each $k \ge k_0$ we
  have
  \begin{displaymath}
    \HM^{m(m+1)}(\Sigma_k) \le (\omega_m (1 + \varepsilon)^m 2^{-k m})^{m+1} \,.
  \end{displaymath}
  Of course we have
  \begin{multline*}
    \int_{\Sigma}
    \sum_{k=K_0}^{k_0-1}
    \int_{\Sigma_k(x_0)}
    \left(
      \frac{\beta(x_0,\diam(x_0,\bar{x}))}{\diam(x_0,\bar{x})}
    \right)^p\ d\HM^m(x_0) d\mu(\bar{x}) \\
    \le \HM^m(\Sigma) (k_0 - K_0) \omega_m^{m+1} (1 + \varepsilon)^{m(m+1)} 2^{-K_0 m(m+1)} 2^{pk_0} < \infty \,,
  \end{multline*}
  so to show that $\E_p(\Sigma)$ is finite it suffices to estimate the sum from
  $k_0$ to $\infty$. Using Lemma~\ref{lem:beta-est} we can write
  \begin{multline*}
    \int_{\Sigma} \sum_{k=k_0}^{\infty}
    \int_{\Sigma_k(x_0)} \left(
      \frac{\beta(x_0,\diam(x_0,\bar{x}))}{\diam(x_0,\bar{x})}
    \right)^p\ d\HM^m(x_0) d\mu(\bar{x}) \\
    \le C \HM^m(\Sigma) \sum_{k=k_0}^{\infty}
      (\omega_m (1 + \varepsilon)^m )^{m+1}
      2^{-k m(m+1)} \Big(\frac{2^{-k\alpha}}{2^{-k-1}}\Big)^p \\
    = C'(m,p,\Sigma,\varepsilon) \sum_{k=k_0}^{\infty} 2^{-k m(m+1) - k p \alpha + kp} \,.
  \end{multline*}
  This sum is finite if and only if
  \begin{displaymath}
    -m(m+1) - p \alpha + p < 0
    \quad \iff \quad
    \alpha > 1 - \frac{m(m+1)}p \,.
  \end{displaymath}
\end{proof}

% Local Variables:
% coding: iso-8859-2
% eval: (ispell-change-dictionary "american")
% mode: latex
% mode: flyspell
% End:

% LocalWords:  submanifolds Grassmannian

%% file: curvEx.tex
\section{Construction of a $C^{1,1-2/p}$ curve with infinite $\M_p$ energy}
\label{sec:curv-ex}

In this section we shall prove the following theorem.
\begin{thm}
  \label{thm:curvEx}
  Let $p > 2$ and set $\alpha = 1 - \frac 2p$. There exists a function $F \in
  C^{1,\alpha}(\R)$ such that $\M_p(\graph(F)) = \infty$.
\end{thm}

The construction of our function is based on the van~der~Waerden saw. Let
\begin{displaymath}
  \tilde f(x) = \left\{
    \begin{array}{cc}
      2x   & \textrm{for} \ x\in\left[ 0,\frac{1}{2} \right] \\
      2-2x & \textrm{for} \  x \in \left( \frac{1}{2},1 \right]
    \end{array}
  \right.
  \quad \text{and} \quad
  f_0(x) = \tilde f(x-[x]) \,.
\end{displaymath}
For $\alpha \in (0,1) $ we define a sequence of functions 
\begin{displaymath}
  f_k(x) = \frac{f_0(N^k x)}{N^{\alpha k}} \,,
\end{displaymath}
where $N \in \N$ is a fixed number whose value be determined later on. We set
\begin{equation}
  \label{def:f}
  f(x) = \sum_{k=0}^\infty f_k(x) \,.
\end{equation}
Finally we define
\begin{equation}
  \label{def:F}
  F(x) = \int_0^x f(t)\ dt = \sum_{k=0}^{\infty} \int_0^x f_k(t)\ dt \,.
\end{equation}

First we show that $F$ is  $C^{1,\alpha}$.
\begin{lem}
  \label{lem:holder}
  The function $F : \R \to \R$ defined by \eqref{def:F} is of class
  $C^{1,\alpha}$. Moreover we have
  \begin{displaymath}
    \forall x,y \in \R
    \quad
    |F'(x) - F'(y)| \le \left( \frac{2N}{N^{1-\alpha} - 1} + \frac{2N^{\alpha}}{N^{\alpha}-1} \right) |x-y|^{\alpha} \,.
  \end{displaymath}
\end{lem}

\begin{proof}
  Since $F'(x) = f(x)$, it suffices to indicate that $f : \R \to \R$ is
  H{\"older} continuous. Note that $f$ is periodic with period $1$, so it is
  enough to show that $|f(x) - f(y)| \lesssim |x-y|^{\alpha}$ only for $x \in
  [0,1]$ and $y \in [0,1]$.

  Fix two numbers $x \in [0,1]$ and $y \in [0,1]$ such that $x < y$. Let $h =
  y-x$ and let $l \in \N$ be such that
  \begin{displaymath}
    \frac 1{N^l} \le h \le \frac 1{N^{l-1}} \,.
  \end{displaymath}
  We can express $x$ and $h$ as infinite sums
  \begin{displaymath}
    x = \sum_{j=0}^{\infty} \frac{x_j}{N^j} 
    \qquad \text{and} \qquad
    y - x = h = \sum_{j=l}^{\infty} \frac{h_j}{N^j} \,,
  \end{displaymath}
  where $x_j, h_j \in \{ 0, 1, \ldots, N-1 \}$ for each $j \in \N$. Now, we
  calculate
  
  \begin{align*}
    |f(y) - f(x)| &= |f(x + h) - f(x)| =
    \left|
      \sum_{k=0}^{\infty} \frac 1{N^{\alpha k}}
      \left(
        f_0\Big( \sum_{j=0}^{\infty} \tfrac{N^k x_j}{N^j} \Big)
        - f_0\Big( \sum_{j=0}^{\infty} \tfrac{N^k x_j + N^k h_j}{N^j} \Big)
      \right)
    \right| \\
    &\le 
    \sum_{k=0}^{\infty} \frac 1{N^{\alpha k}}
    \left|
      f_0\Big( \sum_{j=\max(k+1,l)}^{\infty} \tfrac{N^k x_j}{N^j} \Big)
      - f_0\Big( \sum_{j=\max(k+1,l)}^{\infty} \tfrac{N^k x_j + N^k h_j}{N^j} \Big)
    \right| \,.
  \end{align*}
  Hence, using the fact that $f_0$ is Lipschitz continuous with Lipschitz
  constant $2$ we obtain
  \begin{align*}
    |f(y) - f(x)| &\le 
    \sum_{k=0}^{l-1} \frac 2{N^{\alpha k}} \sum_{j=l}^{\infty} \frac{N^k h_j}{N^j} +
    \sum_{k=l}^{\infty} \frac 2{N^{\alpha k}} \sum_{j=k+1}^{\infty} \frac{N^k h_j}{N^j} \\
    &\le 2 |y-x| \sum_{k=0}^{l-1} N^{k(1-\alpha)} +
    2 \sum_{k=l}^{\infty} \frac{N^k}{N^{\alpha k}} \sum_{j=k+1}^{\infty} \frac{N-1}{N^j} \\
    &= 2 |y-x| \frac{N^{l(1-\alpha)} - 1}{N^{1-\alpha} - 1} +
    \sum_{k=l}^{\infty} \frac 2{N^{\alpha k}}
    \le \frac 2{N^{\alpha l}} \frac{N}{N^{1-\alpha} - 1} + \frac 2{N^{\alpha l}} \frac{N^{\alpha}}{N^{\alpha}-1} \\
    &\le \left( \frac{2N}{N^{1-\alpha} - 1} + \frac{2N^{\alpha}}{N^{\alpha}-1} \right) |y-x|^{\alpha} \,.
  \end{align*}
\end{proof}

Recall that $\alpha = 1 - \frac 2p$. Now we shall prove that the $p$-integral
curvature $\M_p(\graph(F))$ is infinite.
\begin{lem}
  \label{lem:infinite}
  The function $F : \R \to \R$ defined by \eqref{def:F} satisfies
  $\M_p(\graph(F)) = \infty$.
\end{lem}

\begin{proof}
  The graph of $F$ is not a closed curve, thus we express the $\M_p$ energy in the following way
  \[
  \M_p(\graph(F)) := \int_0^L \int_0^L \int_0^L
  R^{-p}(\Gamma(x), \Gamma(y), \Gamma(z))\,dx\,dy\,dz \,,
  \]
  where $\Gamma$ is an arc-length parametrization of $\graph(F)$ and $L = \HM^1(\graph(F))$.
  It is easy to notice that
  \[
  \M_p(\graph(F)) > \int_0^1 \int_0^1 \int_0^1
  R^{-p}((x,F(x)), (y,F(y)), (z, F(z)))\,dx\,dy\,dz \, 
  \]
  
  We denote by $\vec{t}$ a point of the graph given by argument $t$;
  i.e. $\vec{t}=(t,F(t))$. Since $F$ is a~Lipschitz function, for any two points
  of the graph we have
  \begin{displaymath}
    |t-s| \le \|\vec{t}-\vec{s} \| 
    = \sqrt{(t-s)^2 + \big(F(t)-F(s)\big)^2} 
    \le C(\alpha)|t-s| \,.
  \end{displaymath}  
  Let us start with the following estimation of the Menger curvature of three
  points of the graph. If $0\le x \le z \le y \le 1$ we have
  \begin{displaymath}
    \frac{1}{R(\vec{x},\vec{y},\vec{z})} 
    = \frac{4 \HM^2 \big(\conv(\vec{x},\vec{y},\vec{z})\big)}
    {\|\vec{x}-\vec{y} \|\|\vec{y}-\vec{z}\| \|\vec{x}-\vec{z}\| }
    \ge \frac{2h}{\|\vec{x}-\vec{y}\|^2} 
    = \frac{\sin\varangle(\vec{x}-\vec{y}, \vec{x}-\vec{z})}{\|\vec{x}-\vec{y} \|^2} \| \vec{x}-\vec{z}\| \,,
  \end{displaymath}
  where $h$ denotes the height of the triangle $\conv(\vec{x},\vec{y},\vec{z})$
  which is perpendicular to $\vec{x}-\vec{y}$. In order to find a lower bound
  for the above expression we estimate tangent of the angle between
  $\vec{x}-\vec{y}$ and $\vec{x}-\vec{z}$
  \begin{displaymath}
    |\tan(\varangle(\vec{x}-\vec{y}, \vec{x}-\vec{z}))|
    = \frac{\left|\frac{F(y)-F(x)}{y-x}-\frac{F(z)-F(x)}{z-x}\right|}{1+\frac{F(y)-F(x)}{y-x}\frac{F(z) - F(x)}{z-x}}
    \ge C(\alpha) \left| \frac{F(y)-F(x)}{y-x}-\frac{F(z)-F(x)}{z-x} \right| \,.
  \end{displaymath}
  Therefore if $\varangle(\vec{x}-\vec{y},\vec{x}-\vec{z}) \le \frac{\pi}{3}$ we
  have
  \begin{displaymath}
    \frac{1}{R(\vec{x},\vec{y},\vec{z})} 
    \ge \frac{C(\alpha)\left| \frac{F(y)-F(x)}{y-x}-\frac{F(z)-F(x)}{z-x} \right|(z-x)}{(y-x)^2} \,.
  \end{displaymath}
  We will prove that the energy is infinite even when we consider much smaller
  domain of integration. For $k \in \N$ and $m \in \{ 0,1, \ldots, N-1 \}$ we
  define the following intervals
  \begin{align*}
    X_{k,m} &= \left[\frac{m}{N^k},\frac{m}{N^k}+\frac{1}{16}\frac{1}{N^k} \right] \,, \\
    Y_{k,m} &= \left[\frac{m+1/2}{N^k}-\frac{1}{16}\frac{1}{N^k},\frac{m+1/2}{N^k} \right] \,, \\
    Z_{k,m} &= \left[\frac{m}{N^k}+ \frac{1}{4}\frac{1}{N^k},\frac{m}{N^k}+\frac{1}{4N^k}+\frac{1}{16N^k} \right] \,.
  \end{align*}
  Of course, for sufficiently large $N$, $\varangle(\vec{x}-\vec{y},\vec{x}-\vec{z}) \le \frac{\pi}{3}$ and we have
  \begin{displaymath}
    \M_p(\graph(F)) 
    \ge \sum_{k=1}^\infty \sum_{m=0}^{N^k} 
    \int_{X_{k,m}} \int_{Y_{k,m}} \int_{Z_{k,m}}
    R^{-p}(\vec{x},\vec{y},\vec{z})\ dx\ dy\ dz \,,
  \end{displaymath}
  For our purposes it is enough to carry the calculation only for $x \in
  X_{k,m}$, $y \in Y_{k,m}$ and $z \in Z_{k,m}$. Recall that we have
  \begin{displaymath}
    F(x) =  \sum_{k=0}^\infty \int_0^x f_k(t)\ dt \,.
  \end{displaymath}
  For convenience we denote $F_k(x) = \int_0^x f_k(t)$ and we notice that
  \begin{displaymath}
    \left| \frac{F(y)-F(x)}{y-x}-\frac{F(z) -F(x)}{z-x}\right|
    = \left| \sum_{k=0}^\infty \frac{F_k(y)-F_k(x)}{y-x}-\frac{F_k(z) -F_k(x)}{z-x} \right| \,.
  \end{displaymath}
  We need to divide the above sum into two parts which behave differently.

  \begin{itemize}
  \item For $n\le k$ there are two possibilities. Either
    \begin{align*}
      \frac{F_n(y)-F_n(x)}{y-x}
      &= \frac{1}{y-x}\int_{x}^{y} \left( \frac{1}{N^n} \right)^\alpha f(N^nt)\ dt \\
      &= \frac{1}{y-x}\int_{x}^{y} \left(\frac{1}{N^n}\right)^\alpha 2N^nt\ dt
      = 2 \left(\frac{1}{N^n}\right)^{\alpha-1} (x+y)
    \end{align*}
    and
    \begin{displaymath}
      \frac{F_n(z)-F_n(x)}{z-x}= 2 \left(\frac{1}{N^n}\right)^{\alpha-1} (x+z)
    \end{displaymath}
    or
    \begin{align*}
      \frac{F_n(y)-F_n(x)}{y-x}
      &= \frac{1}{y-x} \int_x^y\left(\frac{1}{N^n}\right)^\alpha2(1-N^n t)\ dt \\
      &= 2\left(\frac{1}{N^n}\right)^\alpha - 2\left(\frac{1}{N^n}\right)^{\alpha-1}(x+y)
    \end{align*}
    and
    \begin{displaymath}
      \frac{F_n(z)-F_n(x)}{z-x} 
      = 2\left(\frac{1}{N^n}\right)^\alpha - 2\left(\frac{1}{N^n}\right)^{\alpha-1}(x+z) \,.
    \end{displaymath}
    Denoting by
    \[
    \delta_n(x,z,y) :=  \frac{F_n(y)-F_n(x)}{y-x} - \frac{F_n(z)-F_n(x)}{z-x},
    \] 
    in both cases we have
    \begin{displaymath}
      |\delta_n(x,z,y)|  
      = 2 (N^n)^{1-\alpha}(y-z) \,.
    \end{displaymath}
    Now, taking $N$ sufficiently large:
    \begin{equation}
      \label{nlek-final}
      \left|\sum_{n=1}^k \delta_n(x,z,y) \right| 
      \ge \frac{1}{2} (N^k)^{1-\alpha}(y-z)
      \ge \frac{1}{2}(N^k)^{1-\alpha}\frac{3}{16}\frac{1}{N^k}
      = \frac{3}{32}N^{-k\alpha} \,.
    \end{equation}

  \item For $n>k$ we notice that
    \begin{align*}
      F_n(y) - F_n(x) 
      &= \int_{x}^y f_n(t)\ dt
      \le \left[(y-x)2N^n\right]\frac{1}{4N^n}\left(\frac{1}{N^n}\right)^{\alpha}
      + \frac 1{2N^n} \left(\frac{1}{N^n} \right)^\alpha  \\
      &\le (y-x) \frac 12 \left(\frac{1}{N^n} \right)^\alpha 
      + \frac 12 \left(\frac{1}{N^n} \right)^{\alpha+1}
    \end{align*}
    and
    \begin{displaymath}
      F_n(y) - F_n(x)
      \ge \frac 12 (y-x) \left(\frac{1}{N^n}\right)^\alpha 
      - \frac 14 \left(\frac{1}{N^n}\right)^{\alpha+1} \,.
    \end{displaymath}
    As $y-x \ge \frac 38 \frac 1{N^k}$ we have
    \begin{displaymath}
      \frac 12 \left(\frac{1}{N^n}\right)^{\alpha} 
      - \frac 23 N^k \left(\frac{1}{N^n}\right)^{\alpha+1}
      \le \frac{F_n(y) - F_n(x)}{y-x}
      \le \frac 12 \left(\frac{1}{N^n} \right)^\alpha
      + \frac 43 \left(\frac{1}{N^n}\right)^{\alpha + 1} N^k \,.
    \end{displaymath}
    Analogously
    \begin{displaymath}
      \frac 12 \left(\frac{1}{N^n}\right)^{\alpha} 
      - \frac 43 N^k \left(\frac{1}{N^n}\right)^{\alpha+1}
      \le \frac{F_n(z) - F_n(x)}{z-x}
      \le \frac 12 \left(\frac 1{N^n}\right)^{\alpha} 
      + \frac 83 N^k \left(\frac{1}{N^n}\right)^{\alpha+1} \,.
    \end{displaymath}
    Thus there exists a constant (which does not depend on $k, n$ and $N$), such
    that
    \begin{displaymath}
      |\delta_n(x,z,y)| < CN^k \left( \frac 1{N^n} \right)^{\alpha+1} \,.
    \end{displaymath}
    We can choose $N$ large enough to estimate the geometric series from above by its
    first term
    \begin{displaymath}
      \left|\sum_{n=k+1}^\infty \delta_n(x,z,y)\right|
      < \sum_{n=k+1}^\infty CN^n \left( \frac 1{N^{n}} \right)^{\alpha+1}
      < 2CN^k\left( \frac 1{N^{k+1}} \right)^{\alpha+1} \,,
    \end{displaymath}
    enlarging $N$ if necessary we get
    \begin{equation}
      \label{ngk-final}
      \left|\sum_{n=k+1}^\infty \delta_n(x,z,y)\right|
      < \frac 1{32} \left(\frac 1{N^{k}}\right)^{\alpha} \,.
    \end{equation}
  \end{itemize}
  Putting both estimates \eqref{nlek-final} and \eqref{ngk-final} together
  we~obtain
  \begin{align*}
    \left| \sum_{n=0}^\infty \delta_n(x,z,y) \right| 
    &\ge \left| \sum_{n=0}^k \delta_n(x,z,y) \right| 
    - \left| \sum_{n=k+1}^\infty \delta_n(x,z,y) \right| \\
    &\ge \frac 3{32} \left( \frac 1{N^{k}} \right)^{\alpha} 
    - \frac 1{32} \left( \frac 1{N^{k}} \right)^{\alpha}
    \ge \frac 1{16} \left( \frac 1{N^{k}} \right)^{\alpha} \,.
  \end{align*}
  Thus for each $k,m \in \N$ and $x\in X_{k,m}$, $y\in X_{k,m}$, $z \in Z_{k,m}$
  we have
  \begin{equation}
    \label{est:menger}
    \left|\frac{F(y)-F(x)}{y-x} - \frac{F(x)-F(z)}{x-z}\right|
    \ge \frac{1}{16}\left(\frac{1}{N^{k}}\right)^{\alpha} \,.
  \end{equation} 

  Using \eqref{nlek-final} and \eqref{ngk-final} we can estimate the integral
  \begin{align*}
    \M_p(\graph(F)) 
    &\ge \sum_{k=1}^\infty \sum_{m=0}^{N^k-1} \int_{X_{k,m}} \int_{Y_{k,m}} \int_{Z{k,m}}
    R^{-p}(\vec{x},\vec{y},\vec{z})\ dx\ dy\ dz \\
    &\ge \sum_{m=0}^{N^k-1} \int_{X_{k,m}} \int_{Y_{k,m}} \int_{Z{k,m}}
    C(\alpha) \left(\frac{\frac{1}{16}N^{-k\alpha}(z-x)}{(y-x)^2}\right) \\
    &\ge C(\alpha) \sum_{k=1}^\infty \sum_{m=0}^{N^k-1}
    \left(
      \frac{\frac{1}{16}N^{-k\alpha}\frac{3}{16}N^{-k}}{\frac{1}{4}N^{-2k}}
    \right)^p |X_{k,m}|\cdot|Y_{k,m}| \cdot |Z_{k,m}| \\
    & \ge C(\alpha) \sum_{k=1}^\infty \sum_{m=0}^{N^k-1} 
    (N^k)^{(1-\alpha)p -3} 
    \ge C(\alpha) \sum_{k=1}^\infty (N^k)^{p-p\alpha -2} \,.
  \end{align*}
  Hence, the energy of the graph of $F$ is infinite whenever
  \begin{displaymath}
    p - p\alpha - 2 \ge 0
    \qquad \iff \qquad
    \alpha \le 1 - \frac 2p \,.
  \end{displaymath}
\end{proof}

% Local Variables:
% coding: iso-8859-2
% eval: (ispell-change-dictionary "american")
% mode: latex
% mode: flyspell
% End:

% LocalWords:  Lipschitz

%% file: mfldEx.tex
\section{Higher dimensional case}
\label{sec:mfld-ex}

Here we establish an analogue of Theorem~\ref{thm:curvEx} for the energy $\E_p$.
\begin{thm}
  \label{thm:mfldEx}
  Let $p > m(m+1)$ and set $\alpha = 1 - \frac{m(m+1)}p$. There exists a
  manifold $\Sigma$ of class $C^{1,\alpha}$ such that $\E_p(\Sigma)$ is
  infinite.
\end{thm}

Our construction is based on the same idea as the construction presented in
\S\ref{sec:curv-ex}. Let $N \in \N$ be a big natural number and let $F : \R \to
\R$ be defined by \eqref{def:F}. We define
\begin{align*}
  G : \R^m \to \R^{m+1}
  \quad \text{by the formula} \quad
  G(x^1,\ldots,x^m) = (x^1,\ldots,x^m, F(x^1))
\end{align*}
and we set $\Sigma = G([0,1]^m)$. From Lemma~\ref{lem:holder} it follows that
$\Sigma$ is a $C^{1,\alpha}$ manifold.

To perform the proof of Theorem~\ref{thm:mfldEx} we need to introduce some
additional notation. By $(\pkt{e}_1,\ldots,\pkt{e}_m)$ we denote the standard
basis of $\R^m$. We adopt the convention to typeset points and~vectors
in~$\R^m$ with the bold font $\pkt{x}, \pkt{y}, \pkt{z}$ etc. and to number
the~components with a~superscript, so we shall have $\pkt{x} =
(x^1,\ldots,x^m)$. Also, vectors and points in $\R^{m+1}$ will always be
marked with an arrow $\wpt{x}, \wpt{y}, \wpt{z}$ etc. and we silently assume
that $\pkt{x}$ and $\wpt{x}$ always satisfy
\begin{displaymath}
  \pi_{\R^m}(\wpt{x}) = \pi_{\R^m}(x^1,\ldots,x^{m+1}) = (x^1,\ldots,x^m) = \pkt{x} \,.
\end{displaymath}

Let $U$ and $V$ be any $m$-dimensional subspaces of $\R^{m+1}$. We define
\begin{displaymath}
  Q_U = \pi_{U^{\perp}} = \opid - \pi_U
  \quad\text{and}\quad
  \dgras(U,V) :=  \| \pi_U - \pi_V\| = \| Q_U - Q_V\| \,.
\end{displaymath}

Let $\varepsilon > 0$ be some small constant - its value will be fixed later on.
Let $A \in (0,1)$ be such that $|F'(t)| \le \varepsilon$ whenever $t < A$. For
$n \in \N$ we set
\begin{displaymath}
  J_n = N^{-n} \Z^m \cap [0,A)^m 
  = \left\{
    \left( \tfrac{k_1}{N^n}, \ldots, \tfrac{k_m}{N^n} \right) 
    : \forall i \in \{1,\ldots,m\}\ k_i \in \{0, 1, \ldots, \lceil A N^n - 1 \rceil\}
  \right\} \,.
\end{displaymath}

For any $\pkt{x} \in J_n$ we define
\begin{displaymath}
  \pkt{x}_0^{(\pkt{x})} = \pkt{x} \,,
  \quad
  \pkt{x}_k^{(\pkt{x})} = \pkt{x} + \frac{\pkt{e}_k}{2N^n} \quad \text{for } k = 1,\ldots,m
  \quad\text{and}\quad
  \pkt{x}_{m+1}^{(\pkt{x})} = \pkt{x} + \frac{\pkt{e}_1}{4N^n} \,.
\end{displaymath}
Of course $\pkt{x}_i^{(\pkt{x})}$ depends on $n$ but we do not highlight it in our
notation. Next, we choose a small number $\delta \in (0, \frac 1{16})$ and
we~define
\begin{displaymath}
  U(\pkt{x}) = \Ball^m\big(\pkt{x}_0^{(\pkt{x})}, \tfrac{\delta}{N^n}\big) \times \cdots \times
  \Ball^m\big(\pkt{x}_{m+1}^{(\pkt{x})}, \tfrac{\delta}{N^n}\big) \,.
\end{displaymath}
Note that
\begin{equation}
  \label{eq:Jquant}
  |J_n| = \lceil A N^n \rceil^m \ge A^m N^{nm} \,,
\end{equation}
\begin{equation}
  \label{eq:Udisj}
  \forall \pkt{x} \in J_{n_1}\ \forall \pkt{y} \in J_{n_2}
  \quad \pkt{x} \ne \pkt{y}\ \Rightarrow\ U(\pkt{x}) \cap U(\pkt{y}) = \emptyset
\end{equation}
\begin{equation}
  \label{eq:Umeas}
  \text{and}\quad
  \forall \pkt{x} \in J_n \quad \HM^{m(m+2)}(U(\pkt{x})) 
  = \left( \frac{\omega_m \delta^m}{N^{nm}} \right)^{m+2} \,.
\end{equation}
Let $T = (\pkt{z}_0,\ldots,\pkt{z}_{m+1}) \in (\R^{m})^{m+2}$ and
$G(T)=(G(\pkt{z}_0),\ldots,G(\pkt{z}_{m+1}))=(\wpt{z}_0,\ldots,\wpt{z}_{m+1})$.
We~define
\begin{align*}
  \face(T) &= \conv\{ \wpt{z}_0,\ldots,\wpt{z}_m \} 
  &&\text{- the face of $G(T)$ spanned by } \wpt{z}_0,\ldots,\wpt{z}_m \,, \\
  \plane(T) &= \lin\{ \wpt{z}_1 - \wpt{z}_0, \ldots, \wpt{z}_m - \wpt{z}_0 \}
  &&\text{- the vector space containing } \face(T) - \wpt{z}_0 \\
  \text{and} \quad
  \height(T) &= \dist(\wpt{z}_{m+1}), \wpt{z}_0 + \plane(T))
  &&\text{- the height of $G(T)$ lowered from } \wpt{z}_{m+1} \,.
\end{align*}
Finally we define
\begin{displaymath}
  \DC_G(T) = \DC(G(T)) = \DC(G(\pkt{z}_0),\ldots,G(\pkt{z}_{m+1})) \,.
\end{displaymath}
Note that 
\begin{displaymath}
  |JG(x)| = \sqrt{\det(DG(x)^T DG(x))} \ge 1 \,,
\end{displaymath}
so we have
\begin{equation}
  \label{eq:Ecomp}
  \E_p(\Sigma) \ge \int_{[0,1]^{m(m+2)}} \DC_G(\pkt{z}_0,\ldots,\pkt{z}_{m+1})^p\ d\pkt{z}_0 \cdots\ d\pkt{z}_{m+1} \,.
\end{equation}

\begin{prop}
  \label{prop:h-est}
  There exists $n_0 \in \N$ such that for all $n \ge n_0$, if $\pkt{x} \in J_n$
  and $T = (\pkt{z}_0,\ldots,\pkt{z}_{m+1}) \in U(\pkt{x})$, then we have
  \begin{displaymath}
    \height(T) \ge C N^{-n(1+\alpha)}
    \qquad \text{and} \qquad
    \DC_G(T) \ge \tilde{C} N^{n(1-\alpha)} \,.
  \end{displaymath}
\end{prop}

Using Proposition~\ref{prop:h-est} we can finish the proof of the main theorem.
\begin{proof}[Proof of Theorem~\ref{thm:mfldEx}]
  Using \eqref{eq:Jquant}, \eqref{eq:Udisj} and~\eqref{eq:Umeas} together
  with~\eqref{eq:Ecomp} we get
  \begin{align*}
    \E_p(\Sigma) &\ge \sum_{n=n_0}^{\infty} \sum_{\pkt{x} \in J_n} \int_{U(\pkt{x})}
    \DC_G(T)^p\ d\pkt{z}_0 \cdots\ d\pkt{z}_{m+1} \\
    &\ge \tilde{C} \sum_{n=n_0}^{\infty} A^m N^{nm} \frac{\omega_m^{m+2} \delta^{m(m+2)}}{N^{nm(m+2)}} N^{n(1-\alpha)p} \\
    &= \hat{C}(\delta,m) \sum_{n=n_0}^{\infty} N^{n(m - m(m+2) + (1-\alpha)p)} \,,
  \end{align*}
  which is infinite if and only if
  \begin{displaymath}
    m - m(m+2) + (1-\alpha)p \ge 0
    \quad \iff \quad
    p \ge \frac{m(m+1)}{1-\alpha}
    \quad \iff \quad
    \alpha \le 1 - \frac{m(m+1)}p \,.
  \end{displaymath}
  In particular, for $\alpha = 1 - \frac{m(m+1)}p$ we have $\E_p(\Sigma) =
  \infty$.
\end{proof}

Now we only need to prove Proposition~\ref{prop:h-est}.
\begin{proof}[Proof of Proposition~\ref{prop:h-est}]
  Recall that $T = (\pkt{z}_0,\ldots,\pkt{z}_{m+1}) \in U(\pkt{x})$ and
  $\pkt{x} \in J_n$, so $T$ and $\pkt{x}$ depend on $n$. Note that the vectors
  $\pkt{z}_i - \pkt{z}_0$ for $i=1,\ldots,m$ satisfy
  \begin{align*}
    (1 - 2\delta)\frac 1{2N^n} \le |\pkt{z}_i - \pkt{z}_0| &\le (1 + 2\delta)\frac 1{2N^n} \\
    \text{and} \quad
    |\langle \pkt{z}_i - \pkt{z}_0, \pkt{z}_j - \pkt{z}_0 \rangle| &\le \frac{3 \delta}{(2N^n)^2} \,.
  \end{align*}
  Hence, decreasing $\delta$ we can make the vectors $\pkt{z}_i - \pkt{z}_0$ for
  $i=1,\ldots,m$ form a roughly orthogonal basis of $\R^m$. Observe that
  $\varepsilon > 0$ can be chosen as small as we want and recall that we have
  $|F'| \le \varepsilon$. Therefore, we can find $\varepsilon$ such that
  for~each $n$ we have (see~\cite[\S1.3]{slawek-phd} for the proof)
  \begin{equation}
    \label{est:angle}
    \dgras(\plane(T), \R^m)
    = \|\pi_{\plane(T)} - \pi_{\R^m}\|
    = \|Q_{\plane(T)} - Q_{\R^m}\| \le \frac 12 \,.
  \end{equation}
  Note that the vector $\pkt{z}_{m+1} - \pkt{z}_0$ lies in $\R^m$, so it can
  be expressed as
  \begin{displaymath}
    \pkt{z}_{m+1} - \pkt{z}_0 = \sum_{i=1}^m \zeta_i (\pkt{z}_i - \pkt{z}_0) \,,
  \end{displaymath}
  for some $\zeta_1$, \ldots, $\zeta_m$. Observe that when we decrease $\delta$
  to zero, the point $\pkt{z}_{m+1}$ approaches the midpoint $\frac 12
  (\pkt{z}_0+\pkt{z}_1)$, so $\zeta_1$ converges to $\frac 12$ and all the other
  $\zeta_i$ for $i = 2,\ldots,m$ converge to $0$. The values $|\zeta_1 - \frac
  12|$ and $|\zeta_i|$ can be bounded above independently of the scale we are
  working in (i.e. independently of the choice of $n$) and also independently of
  the choice of $T$ in $U(\pkt{x})$. Hence
  \begin{equation}
    \label{eq:zeta-coeff}
    \left|\zeta_1 - \frac 12\right| \le \Xi(\delta)
    \quad\text{and}\quad
    |\zeta_i| \le \Xi(\delta) \quad \text{for } i = 2,\ldots,m \,,
    \quad\text{where}\ \ 
    \Xi(\delta) \xrightarrow{\delta \to 0} 0
  \end{equation}
  and $\Xi(\delta)$ is a term depending only on $\delta$ and $m$.

  For $i = 1,\ldots,m+1$ we set 
  \begin{displaymath}
    \wpt{z}_i = G(\pkt{z}_i)
    \quad \text{and we define} \quad
    \wpt{w} = \wpt{z}_0 + \sum_{i=1}^m \zeta_i (\wpt{z}_i - \wpt{z}_0) \in \wpt{z}_0 + \plane(T) \,.
  \end{displaymath}
  Then
  \begin{displaymath}
    \pkt{z}_{m+1} = \pi_{\R^m}(\wpt{z}_{m+1}) = \pi_{\R^m}(\wpt{w})
    \quad\text{and}\quad
    (\wpt{z}_{m+1} - \wpt{w}) \perp \R^m 
  \end{displaymath}
  \begin{displaymath}
    \text{and also}\quad
    Q_{\plane(T)}(\wpt{z}_{m+1} - \wpt{w}) = Q_{\plane(T)}(\wpt{z}_{m+1} - \wpt{z}_0) \,.
  \end{displaymath}
  Therefore, using the angle estimate \eqref{est:angle} we obtain
  \begin{align}
    \label{est:hwz}
    \height(T) &= \dist(\wpt{z}_{m+1}, \wpt{z}_0 + \plane(T))
    = |Q_{\plane(T)}(\wpt{z}_{m+1} - \wpt{z}_0)| \\
    &= |Q_{\plane(T)}(\wpt{z}_{m+1} - \wpt{w})|  
    = |Q_{\R^m}(\wpt{z}_{m+1} - \wpt{w}) - (Q_{\R^m} - Q_{\plane(T)})(\wpt{z}_{m+1} - \wpt{w})| \notag \\
    &\ge |\wpt{z}_{m+1} - \wpt{w}| - \frac 12 |\wpt{z}_{m+1} - \wpt{w}| 
    = \frac 12 |\wpt{z}_{m+1} - \wpt{w}| \notag \,.
  \end{align}
  Hence, to get a lower bound on $\height(T)$ it suffices to estimate
  $|\wpt{z}_{m+1} - \wpt{w}|$ from below. We calculate
  \begin{align}
    |\wpt{z}_{m+1} - \wpt{w}| 
    &= \left|
      \big( \pkt{z}_{m+1},F(z_{m+1}^1) \big) 
      - \big( \pkt{z}_{m+1}, F(z_0^1) + \sum_{i=1}^m \zeta_i (F(z_i^1) - F(z_0^1)) \big)
    \right| \notag \\
    &= \left|
      F(z_{m+1}^1) -  F(z_0^1) + \sum_{i=1}^m \zeta_i (F(z_i^1) - F(z_0^1))
    \right| \notag \\
    \label{est:height}
    &\ge 
    \left| F(z_0^1 + \zeta_1 (z_1^1 - z_0^1)) -  F(z_0^1) - \zeta_1 (F(z_1^1) - F(z_0^1)) \right| \\
    &\phantom{=\Bigg| }
    - \Bigg|
    \left(
    F\big(z_0^1 + \zeta_1 (z_1^1 - z_0^1) + \sum_{i=2}^m \zeta_i (z_i^1 - z_0^1)\big) 
    - F(z_0^1 + \zeta_1 (z_1^1 - z_0^1))
    \right) \notag \\
    &\phantom{=\Bigg|-\Bigg| }
    - \sum_{i=2}^m \zeta_i (F(z_i^1) - F(z_0^1)) \Bigg| \notag \,.
  \end{align}
  Using \eqref{eq:zeta-coeff} one can estimate the term
  \begin{displaymath}
    \left| F(z_0^1 + \zeta_1 (z_1^1 - z_0^1)) -  F(z_0^1) - \zeta_1 (F(z_1^1) - F(z_0^1)) \right|
  \end{displaymath}
  exactly the same way as in the one dimensional case (cf. \eqref{est:menger}
  with $z_0^1$, $z_1^1$ and $z_0^1 + \zeta_1 (z_1^1 - z_0^1)$ playing the
  roles of $x$, $y$ and $z$ respectively and with $\frac{z-x}{y-x}=\zeta_1$)
  and then obtain
  \begin{equation}
    \label{est:main-part}
    \left| F(z_0^1 + \zeta_1 (z_1^1 - z_0^1)) -  F(z_0^1) - \zeta_1 (F(z_1^1) - F(z_0^1)) \right|
    \ge C |z_1^1 - z_0^1|^{1+\alpha} 
    \ge \tilde{C} N^{-n(1+\alpha)} \,.
  \end{equation}
  We estimate the remaining terms using the mean value theorem for $F$. First we
  have
  \begin{equation}
    \label{eq:mv-fst}
    F\big(z_0^1 + \zeta_1 (z_1^1 - z_0^1) + \sum_{i=2}^m \zeta_i (z_i^1 - z_0^1)\big) 
    - F(z_0^1 + \zeta_1 (z_1^1 - z_0^1))
    = F'(\xi_{m+1}) \sum_{i=2}^m \zeta_i (z_i^1 - z_0^1) \,,
  \end{equation}
  where $\xi_{m+1}$ is some number satisfying
  \begin{displaymath}
    z_{m+1}^1 - \frac{2m\delta}{N^n} < \xi_{m+1} < z_{m+1}^1 + \frac{2m\delta}{N^n}
  \end{displaymath}
  In the same way we obtain
  \begin{equation}
    \label{eq:mv-rest}
    \sum_{i=2}^m \zeta_i (F(z_i^1) - F(z_0^1))
    = \sum_{i=2}^m F'(\xi_i) \zeta_i (z_i^1 - z_0^1) \,,
  \end{equation}
  where $\xi_i \in (z_0^1-\frac{2\delta}{N^n},z_0^1+\frac{2\delta}{N^n})$. Note
  that
  \begin{displaymath}
    |\xi_{m+1} - \xi_i| \le C |z_{m+1}^1 - z_i^1| 
    \le C' N^{-n} \,,
  \end{displaymath}
  \begin{displaymath}
    |z_i^1 - z_0^1| \le \bar{C} \delta N^{-n}
    \quad \text{and} \quad
    |F'(\xi_i) - F'(\xi_{m+1})| \le \hat{C} |\xi_i - \xi_{m+1}|^{\alpha}
    \le \tilde{C} N^{-n\alpha} \,.
  \end{displaymath}
  Putting \eqref{eq:mv-fst} and \eqref{eq:mv-rest} together we get
  \begin{multline}
    \label{est:error}
    \Bigg|
    \left(
    F(z_0^1 + \zeta_1 (z_1^1 - z_0^1) + \sum_{i=2}^m \zeta_i (z_i^1 - z_0^1)) 
    - F(z_0^1 + \zeta_1 (z_1^1 - z_0^1))
    \right) \\
    - \sum_{i=2}^m \zeta_i (F(z_i^1) - F(z_0^1)) \Bigg|
    = \left| \sum_{i=2}^m (F'(\xi_i) - F'(\xi_{m+1})) \zeta_i (z_i^1 - z_0^1) \right|
    \le \hat{C} \delta N^{-n(1+\alpha)} \,.
  \end{multline}
  Plugging \eqref{est:main-part} and \eqref{est:error} into \eqref{est:height}
  we obtain
  \begin{displaymath}
    |\wpt{z}_{m+1} - \wpt{w}| 
    \ge \tilde{C} N^{-n(1+\alpha)} - \hat{C} \delta N^{-n(1+\alpha)} \,.
  \end{displaymath}
  Since we have a freedom in choosing $\delta$, we can make $\hat{C} \delta$ as
  small as we want, so recalling~\eqref{est:hwz} the first part of
  Proposition~\ref{prop:h-est} is proven. The second part follows from a simple
  calculation
  \begin{align*}
    \DC_G(T) &= \frac{\HM^{m+1}(\conv(G(T)))}{\diam(G(T))^{m+2}}
    = \frac{\height(T) \HM^m(\face(T))}{(m+1)\diam(G(T))^{m+2}} \\
    &\ge C \frac{\height(T) \HM^m(\pi_{\R^m}(\face(T)))}{\diam(T)^{m+2}}
    \ge \tilde{C} \frac{N^{-n(1+\alpha)} N^{-nm}}{N^{-n(m+2)}}
    = \tilde{C} N^{n(1-\alpha)} \,.
  \end{align*}
\end{proof}

% Local Variables:
% coding: iso-8859-2
% eval: (ispell-change-dictionary "american")
% mode: latex
% mode: flyspell
% End: